\documentclass[12pt]{article}
	\usepackage{amssymb,amsfonts,amsmath,amsthm, breqn, color, float, mathtools, url, tikz, algorithmic}
	\usetikzlibrary{arrows}
	\usetikzlibrary{automata,positioning}
	\usepackage{caption}
	\usepackage[margin=1.2in]{geometry}

	\DeclareMathOperator\supp{supp}

	\newcommand{\pp}{{\mathbb P}}

	\newcommand{\nn}{{\mathbb N}}

	\newcommand{\rr}{{\mathbb R}}

	\newcommand{\cala}{{\mathcal A}}

	\newcommand{\calg}{{\mathcal G}}

	\newcommand{\calb}{{\mathcal B}}

	\newcommand{\calp}{{\mathcal P}}

	\newcommand{\beq}{\begin{eqnarray*}}
		\newcommand{\feq}{\end{eqnarray*}}
	\newcommand{\beqn}{\begin{eqnarray}}
	\newcommand{\feqn}{\end{eqnarray}}

	\newtheorem{theorem}{Theorem}
	\newtheorem*{conj*}{Conjecture}
	\makeatletter \@addtoreset{theorem}{section}\makeatother
	
	\newtheorem{lemma}[theorem]{Lemma}
	
	\newtheorem*{theorema*}{Theorem~A}
	\newtheorem*{theoremb*}{Theorem~B}
	\newtheorem*{theoremc*}{Theorem~C}
	\newtheorem*{theoremd*}{Theorem~D}
	\newtheorem*{theoreme*}{Theorem~E}
	\newtheorem*{theoremf*}{Theorem~F}
	\newtheorem*{cld*}{Condition $\mbox{LD}_d$}
	\newtheorem*{theorem*}{Theorem}

	\usepackage{graphicx}
	\def\BState{\State\hskip-\ALG@thistlm}
	\makeatother

	\DeclareCaptionLabelFormat{noname}{#2}
	\newlength\myindent
	\DeclareMathOperator{\occ}{\mbox{occ}}

	\title{Assorted inequalities for pattern occurrences}
	
	\author{Reza~Rastegar\thanks{Center of Excellence for Data Science and Modeling, Occidental Petroleum Corporation, Houston, TX 77046 and Departments of Mathematics and Engineering, University of Tulsa, OK 74104, USA - Adjunct Professor; e-mail:   reza.j.rastegar@gmail.com} }
	
\begin{document}
	\maketitle
	\begin{abstract}
In this note, we present several inequalities in the context of pattern containment, utilizing elementary applications of the Fortuin-Kasteleyn-Ginibre (FKG) inequality and Shearer's lemma.
	\end{abstract}
		
	\section{Introduction}

    In this manuscript, we establish several inequalities concerning pattern containment in permutations that has been the subject of intense research over the last three decades, and many extensions and related results can be found in \cite{Bbook, Kbook}. While some of our results have a broader scope, we limit our discussion to permutations for ease of presentation. \par
    
    To present our findings,  a few notations and definitions will be given in the next several paragraphs. We use the notation $\nn:=\{1,2,3,\ldots \}$ and $\nn_0$ to denote the set of natural numbers and the set of non-negative integers, respectively; that is, $\nn_0=\nn \cup \{0\}.$ We also define $\rr_+$ as the set of all non-negative real numbers. Given a set $A,$ we denote its cardinality as $\#A.$ For $n\in \nn$, we set $[n]:=\{1,\cdots, n\}$, and for $d\leq n,$ we define $[n]_d$ to be the set of all $d$-subsets of $[n]$. Moreover, we define $S_n$ as the set of all permutations of length $n$. We will interchangeably interpret any permutation $\pi$ as either a sequence or a vector. In the context of the latter, we use $\pi(A)$ to denote the sequence $(\pi_i)_{i\in A},$ where $A\subset[n]$.

A \emph{pattern} of length $d\in \nn$ is any distinguished permutation chosen from $S_d$. Recall that any word of length $d$ with $d$ distinct letters can be naturally reduced to a permutation in $S_d$, preserving the relative order of the values. For instance, the word $284$ reduces to the permutation $132$. Given an arbitrary permutation $\pi\in S_n$ with $n\geq d$, an occurrence of the pattern $v\in S_d$ in $\pi$ is a sequence of $d$ indices $1\leq j_1<j_2<\dots<j_d\leq n $ such that the \emph{subsequence} $\pi_{j_1}\cdots \pi_{j_d}$ is \emph{order-isomorphic} to the pattern $v$. In other words, $\pi({j_1,\cdots, j_d})$ can be naturally reduced to $v$, as defined by
\beq
\pi_{j_p}<\pi_{j_q}\Longleftrightarrow v_p<v_q\qquad \forall,1\leq p,q\leq d.
\feq
	
For any permutation $\pi\in S_n$ and any pattern $v\in S_d$ with $d\leq n$, we define $\calb_\pi(v)$ as the subset of $[n]_d$ at which $v$ occurs in $\pi$, that is, $\pi(B)$ is order-isomorphic to $v$ if and only if $B\in \calb_\pi(v)\subset [n]_d$. We use $\occ_v(\pi)$ to denote the number of occurrences of $v$ in $\pi$, that is, $\occ_v(\pi) := \# \calb_\pi(v)$. For any $r\in \nn_0$, we denote by $F_r^v(S_n)$ the set of permutations in $S_n$ containing $v$ exactly $r$ times, i.e.,
\beq
F_r^v(S_n)=\{\pi \in S_n:\ \occ_v(\pi)=r\}.
\feq
We define $f_r^v(S_n)$ to be the cardinality of $F_r^v(S_n)$.
For instance, if $v$ is the inversion $21$ and $\pi=12435$, then $\calb_\pi(v)=\{34\}$, $\occ_v(\pi)=1$, and $\pi\in F_1^{21}(S_5)$. 
	
For any permutation $\pi\in S_n$, we denote by $C_d(\pi)$ the set of all patterns $v\in S_d$ that occurs in the permutation $\pi$. For instance:
	\beq
	C_3(1324) = \{132, 123, 213\} \quad \mbox{and}\quad C_3(1234)=\{ 123\}.
	\feq
	In addition, $C_1(\pi)=\{1\}$ for any permutation $\pi.$ Set $c_d(\pi):=\#C_d(\pi)$. 
 
    For our first result, we seek to express $c_d(\pi)$ and $\occ_\pi(v)$ in terms of simpler structures of shorter length for a fixed permutation $\pi$. It states
    
    \begin{theorem} \label{thm1}
    Let $\pi\in S_n$ be a fixed permutation, then
    \begin{itemize}
    \item[(a)] for any $1\leq \ell < d<n$, 
    \beq
    c_d(\pi) \leq \left( c_\ell(\pi) \binom{d}{\ell} \right)^{\frac{d}{\ell}}.
    \feq
    \item[(b)] for any $v\in S_d$ and $1<\ell<d$
    \beq
    \occ_\pi(v) &\leq&  \prod_{w\in C_\ell(v)} \occ_\pi(w)^{\frac{\occ_v(w)}{\binom{d-1}{l-1}}}.
    \feq
    \end{itemize}
    \end{theorem}
    
    As a simple trivial example for part (a), when $\pi$ is the identity permutation and $v$ is the pattern $1\cdots d$, we have $C_{d-1}(v)=\{1\cdots (d-1)\}$, and the inequality reduces to
    \beq
    \binom{n}{d} \leq \binom{n}{d-1} ^{d/(d-1)}.
    \feq 
    With regard to part (b), to the best of our knowledge, there is currently no known simple relation between $\occ_\pi(v)$ and $\occ_\pi(w)$ for $w\in C_\ell(v)$. Therefore, Theorem~\ref{thm1}-(b) offers some new insight in this direction. Figure \ref{fig:ineq1_pic} displays a simulation of the left-hand side and the right-hand side of the inequality. 
    
    \begin{figure}[!h]
	\centering
	\begin{minipage}{.5\textwidth}
		\centering
		\includegraphics[scale=.4]{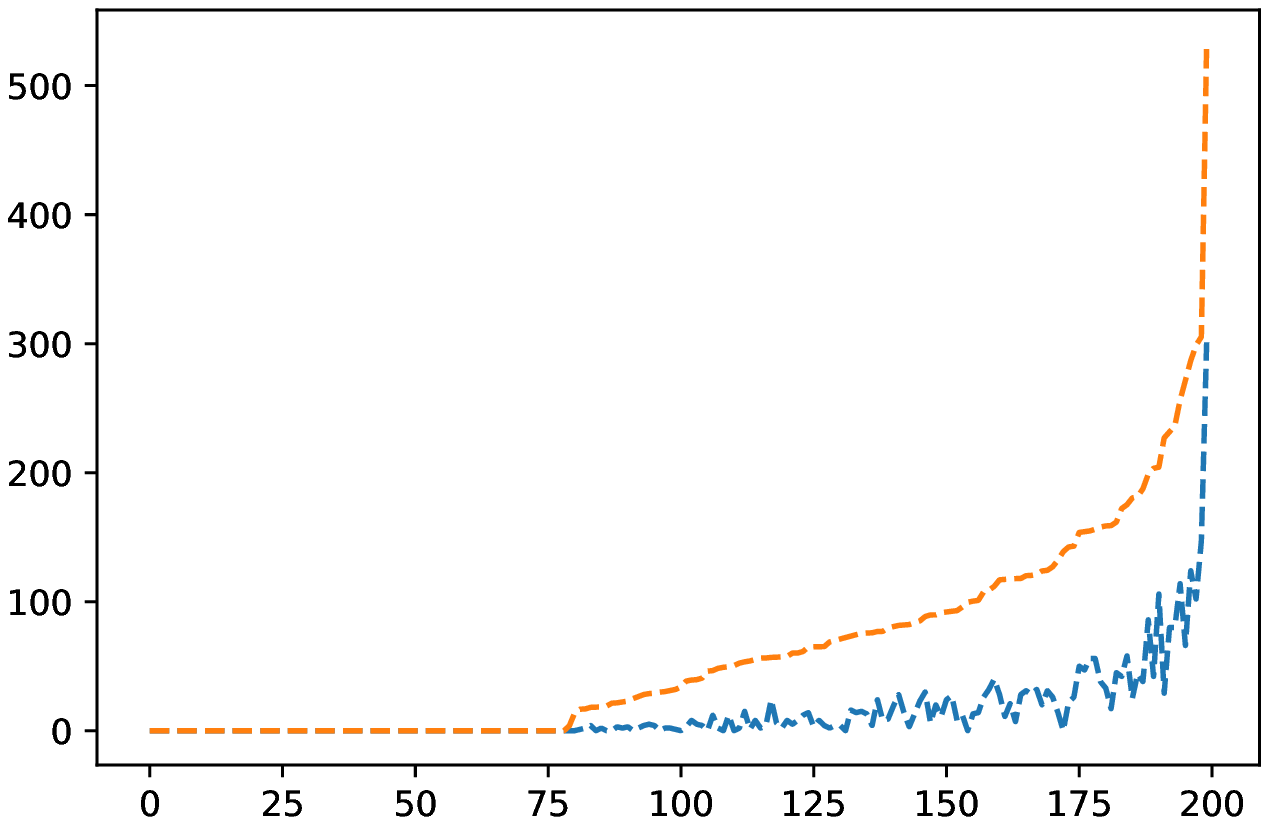}
	\end{minipage}%
	\begin{minipage}{.5\textwidth}
		\centering
		\includegraphics[scale=0.4]{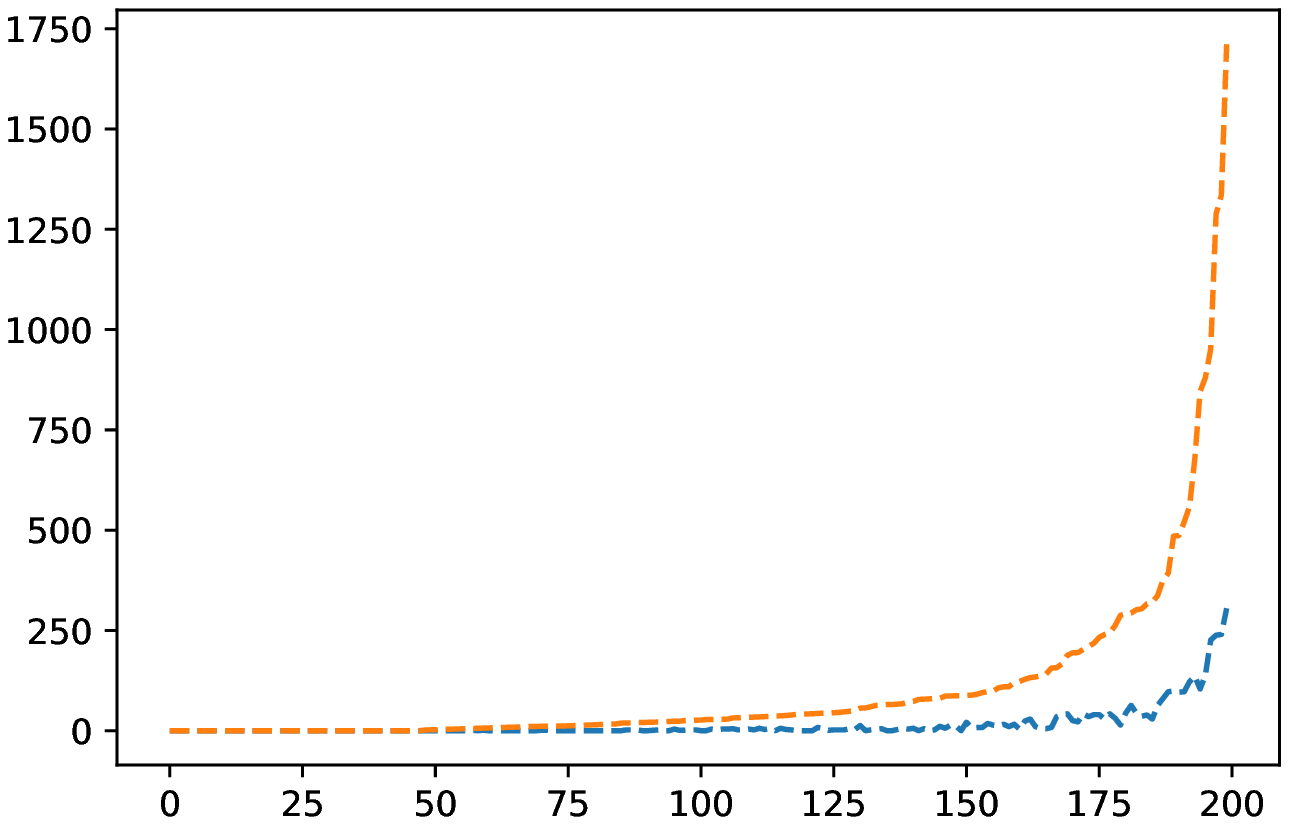}
	\end{minipage}
	\caption{Simulation results comparing the exact value of $\occ_\pi(v)$ (blue plot) with the upper bound predicted by Theorem~\ref{thm1}-(b) (red plot) for 200 randomly generated permutations in $S_{20}$. The pattern $v$ is taken to be $5274316$ on the left and $1234765$ on the right.}
	\label{fig:ineq1_pic}
    \end{figure}  

    The proofs of both parts of the theorem are obtained through entropy argument and an application of Shearer's lemma and are presented in the section \ref{sec1}. \par
Our second result has a different flavor. In order to state it, we first provide an additional definition. Let $B_1,\cdots,B_r$ denote $r$ distinct sets in $[n]d$, and let $F^v_{\{B_1,\cdots,B_r\}}(S_n)$ denote the set of permutations in $S_n$ where order-isomorphic copies of $v$ appear precisely at sets $B_i$. That is, $\pi\in F^v_{\{B_1,\cdots,B_r\}}(S_n)$ if and only if $\calb_\pi(v)=\{B_1,\cdots, B_r\}$. Set $f^v_{\{B_1,\cdots,B_r\}}(S_n):=\#F^v_{\{B_1,\cdots,B_r\}}(S_n)$.

Furthermore, we recall that in the realm of order theory, a function $\mu:L\to\rr_+$ is log-supermodular for any finite distributive lattice $L$ if it satisfies the following inequality:
\beq
\mu(x)\mu(y) \leq \mu(x \wedge y)\mu(x \vee y), \quad x,y\in L,
\feq
where $x \wedge y$ and $x \vee y$ denote the infimum and supremum of $x$ and $y$, respectively, as defined by the order on $L$.Let $\mathcal{P}(n)$ denote the set of all subsets of $[n]$. This set forms a distributive lattice. We define a log-supermodular function on $\mathcal{P}(n)$ using the probability measure $\mu(A) := p^{|A|}(1-p)^{n-|A|}$ for $A\in \mathcal{P}(n)$, where $p\in(0,1)$ is a constant. Our result indicate that it is feasible to derive a correlation inequality between the sizes of sets of permutations categorized based on their fixed containment locations.

    \begin{theorem} \label{lemma1}
    \begin{itemize}
        \item[(a)] Let $\mu$ be any log-supermodular probability measure on $\calp(n)$. Define $$\nu_{\calb}:= \mu\left(  A\in \calp(n) | \forall B\in \calb, B \not\subset A \right), \quad \mbox{for } \calb\subset \calp(n).$$ Then, for each pattern $v\in S_d$, we have 
        $$\prod_{r=1}^{\binom{n}{d}} \prod_{B_1,\cdots, B_r \in [n]_d} \nu_{\{B_1,\cdots,B_r\}}^{f_{\{B_1,\cdots,B_r\}}^v(S_n)} \leq \sum_{A\subset [n], |A|<d} \mu(A).$$ 
        \item[(b)] Choose a chain of subsets $\{\}=A_0\subsetneq A_1 \subsetneq \cdots \subsetneq A_d=[d].$ Let $\mu$ be any probability measure on this chain. For any pattern $v\in S_d,$ and for each $i$, let the pattern $v^i$ to be the reduced form of $v(A_i).$ Then, for any $0 < x_d \leq \cdots \leq x_2 < 1$ we have
        \beq
        \prod_{i=2}^d \left(\sum_{\ell=0}^{i-1} \mu(A_\ell) + \sum_{\ell=i}^d x_\ell\mu(A_\ell)\right)^{f_0^{v^i|v^{i-1}}(S_n)} \leq  \sum_{i=0}^d \mu(A_i) \prod_{\ell=2}^{i} x_\ell^{f_0^{v^{\ell}|v^{\ell-1}}(S_n)},
        \feq
        where $f^{v^i|v^{i-1}}(S_n)$ is the number of permutations in $S_n$ avoiding $v^{i}$ while containing $v^{i-1}.$
    \end{itemize}
    \end{theorem}
    \par

     We use the FKG inequality (Theorem~6.2.1, \cite{noga2}) to prove Theorem~\ref{lemma1} in Section~\ref{kn-sec}. It states that increasing events on these lattices are positively correlated, while an increasing and a decreasing event are negatively correlated.

We point out that the scope of this theorem extends beyond pattern occurrences in permutations to a broader class of sequences. Additionally, we note that certain sets $B_1, \cdots, B_r$ may have $f_{\{B_1,\cdots,B_r\}}^v(S_n)=0$, which can be inferred from the structure of the pattern $v$. Notably, for highly ordered patterns such as $v=1\cdots d$, the inequality (a) may not yield much information. However, we conjecture that the inequality holds stronger for highly unordered patterns, although we currently lack proof for this hypothesis.

    As an illustration of Theorem~\ref{lemma1}-(b), consider the pattern $v=143265$ in $S_6$. We choose the subset chain $$\{\}\subset \{2\}\subset\{2,3\}\subset\{2,3,5\}\subset\{2,3, 5, 6\}\subset\{2,3,4,5,6\}\subset\{1,2,3,4,5,6\}.$$
    The corresponding patterns $v^i$ are $v^{0} = \epsilon$, $v^{1}=1$, $v^{2}=21$, $v^{3}=213$, $v^{4}=2143$, $v^{5}=32154$, and $v=v^6= 143265$, where $\epsilon$ denotes the null permutation. See also the discussion immediately following the proof in the Section \ref{kn-sec}.
  
  \section{Proof of Theorem~\ref{thm1}} \label{sec1}

In combinatorics, entropy based arguments have been extensively used to provide simple yet elegant proof of nontrivial results. See \cite{noga2} and \cite{galvin} and the references within for a review of the method and several interesting examples. In this section, we use entropy to prove Theorem~\ref{thm1}.  To that goal, let $X$ be a random variable sampled from the set $\Omega = \{x_1, . . . , x_m\}$ according to the probability measure $\pp(.)$. Define the entropy of the random variable $X$ as
\beq
H(X) := -\sum_{i=1}^m \pp(X=x_i) \log \pp(X=x_i).
\feq
Entropy has many elegant properties, two of which we will use in the rest of this note: boundedness and sub-additivity. The boundedness property for the entropy of a random variable $X$ is the inequality $H(X)\leq \log \# \supp(X)$, where $\supp X$ is the range of the variable $X$ (see \cite{noga2} - Lemma~15.7.1-(i).)  With respect to the latter, a simple generalization of sub-additivity property (Shearer's lemma - see \cite{noga2} - Proposition~15.7.4) is the main ingredient of our proof. Let $X:=(X_1, \cdots, X_n)$ be any random vector. Shearer's lemma states that for a family of subsets of $[n]$ possibly with repeats, namely $\cala$, with each $i \in [n]$ included in at least $t$ members of $\cala$,  
\beq
t H(X) \leq \sum_{A\in \cala} H(X(A)).
\feq

We use these two properties along with the inequality of arithmetic and geometric means (AM-GM) to bound the quantity that we would like to enumerate.   

\begin{proof}[Proof of Theorem~\ref{thm1}-(a)]
Let  $1\leq \ell < d<n$ be fixed. Fix a permutation $\pi\in S_n.$  
Let $v = (v_1,\cdots,v_d)$ be a uniformly random pattern sampled from $C_d(\pi)$. Suppose $\cala=[d]_{\ell}$ and use $H_\pi(v)$ (resp. $H_\pi(v(A)$) to refer to the entropy of $v$ (resp. $v(A)$). Then, Shearer's lemma implies 
\beqn \label{sh_inst1}
\binom{d-1}{\ell-1} H_\pi(v) \leq \sum_{A\in \cala} H_\pi(v(A)).
\feqn
Since $v$ is uniformly chosen from $C_d(\pi)$, then by the definition of entropy 
\beqn \label{thm1_ine8}
H_\pi(v) = \log c_{d}(\pi).
\feqn
In addition, by the boundedness property we write
\beqn \label{thm1_ine9}
H_\pi(v(A)) \leq \log \#\supp{v(A)}.
\feqn
Plugging \eqref{thm1_ine8} and \eqref{thm1_ine9} into \eqref{sh_inst1}, we get 
\beqn
\log c_d(\pi) &\leq& \frac{1}{\binom{d-1}{\ell-1}}\sum_{A\in \cala} \log \#\supp{v(A)} = \log \left(\prod_{A\in [d]_{\ell}} \#\supp{v(A)}\right)^{\frac{1}{\binom{d-1}{\ell-1}}} \notag \\
&=& \log \left( \prod_{A\in [d]_{\ell}}  c_d(\pi; A) \right)^{\frac{1}{\binom{d-1}{\ell-1}}}, \label{eq30}
\feqn
with $c_d(\pi; A)=\# C_d(\pi; A)$. Here $C_d(\pi; A)$ is a subset of $C_d(\pi)$ such that for any two distinct patterns $v$ and $w$ in $S_d$
\beq
v, w\in C_d(\pi;A) \mbox{ if and only if } v(A)\neq w(A).
\feq
Now, for any $\eta\in S_\ell$ we define $C_d^\eta(\pi;A)$ to be the subset of $C_d(\pi; A)$ such that for any two distinct patterns $v$ and $w$; $v,w\in C^{\eta}_d(\pi;A)$ if and only if
\beq
v(A)\neq w(A), \mbox{ and } v(A),w(A) \mbox{ are order-isomorphic to } \eta.
\feq
Next, we write
\beqn \label{eq20}
C_d(\pi; A) = \cup_{\eta\in C_\ell(\pi)} C_d^\eta(\pi; A).
\feqn
For any $\eta\in C_\ell(\pi),$ there are $\binom{d}{\ell}$ distinct choices for the value of $v(A)$ and hence $\#C_d^\eta(\pi; A)\leq \binom{d}{\ell}$ and 
\beq
c_d(\pi; A) \leq \sum_{\eta\in C_\ell(\pi)} \# C_d^\eta (\pi; A) \leq \binom{d}{\ell} c_\ell(\pi).
\feq
Therefore, \eqref{eq30} yields
\beq
c_d(\pi) &\leq& \left( \max_{A\in [d]_\ell} c_d(\pi; A) \right)^{\frac{\binom{d}{\ell}}{\binom{d-1}{\ell-1}}} \\
&\leq& \left( c_\ell(\pi) \binom{d}{\ell} \right)^{\frac{\binom{d}{\ell}}{\binom{d-1}{\ell-1}}} =  \left( c_\ell(\pi) \binom{d}{\ell} \right)^{\frac{d}{\ell}}.
\feq
This completes the proof.
\end{proof}

\begin{proof} [Proof of Theorem~\ref{thm1}-(b)]
Fix a permutation $\pi\in S_n,$ and a pattern $v\in S_d$. Let $\sigma=(\sigma_1, \cdots, \sigma_d)$ be a uniformly random element chosen from $\calb_\pi(v).$ This implies
\beqn \label{eq1_3}
H_\pi(\sigma)=\log  \occ_\pi(v).
\feqn
We then write 
\beq
&& \binom{d-1}{\ell-1} \log  \occ_\pi(v) = \binom{d-1}{\ell-1} H_\pi(\sigma) \qquad \qquad \textrm{By \eqref{eq1_3}}\\
&& \qquad \leq  \sum_{B\in [d]_\ell} H_\pi(\sigma(B))   \qquad \qquad \qquad \textrm{By Shearer's Lemma} \\
&& \qquad \leq  \sum_{B\in [d]_\ell} \log (\# \supp \sigma(B)) \qquad  \quad  \textrm{By boundedness property} \\
&&\qquad  =  \log  \prod_{B\in [d]_\ell} \# \supp \sigma(B)   \\
&& \qquad = \log  \prod_{w\in C_\ell(v)} \prod_{B\in \calb_v(w)} \# \supp \sigma(B) \\
&&\qquad \leq \log  \prod_{w\in C_\ell(v)} \left( \frac{1}{\#\calb_v(w)} \sum_{B\in \calb_v(w)} \# \supp \sigma(B)\right)^{\#\calb_v(w)}  \  \textrm{By AM-GM} \\
&&\qquad \leq \log \prod_{w\in C_\ell(v)} \occ_\pi(w)^{\occ_v(w)}, 
\feq
where for the last line we observe that, for any $w\in C_\ell(v),$ if $B\in \calb_v(w)$, then $\pi(\sigma(B))$ is an occurrence instance of $w$ in $\pi,$ and hence,
\beqn \label{eq1_1}
\sum_{B\in \calb_v(w)} \# \supp \sigma(B) \leq  \occ_v(w)\occ_{\pi}(w).
\feqn
Divide both sides by $\binom{d-1}{\ell-1}$ and simplify. This completes the proof.
\end{proof}

\section{Proof of Theorem~\ref{lemma1}} \label{kn-sec}

The Fortuin–Kasteleyn–Ginibre (FKG) inequality is a fundamental correlation inequality in statistical physics and percolation, expressed in terms of log-supermodular probability measures on distributive lattices. In this section we use it to prove Theorem~\ref{lemma1}. We first review the inequality FKG. \par Suppose $L$ is a distributive lattice and $\mu$ is a log-supermodular probability measure on $L.$ A non-negative function $g:L\to\rr$ is increasing (resp. decreasing) on $L$ if for every $x\leq y$ we have $g(x)\leq g(y)$ (resp. $g(x)\geq g(y)$). The FKG inequality (See \cite{noga2}, Theorem~6.2.1) states that for family $\calg$ of increasing functions, 
\beqn \label{fkg_ineq}
\prod_{g \in \calg} \sum_{x\in L} \mu(x) g(x) \leq \sum_{x\in L} \mu(x) \prod_{g\in \calg}g(x).
\feqn
The same inequality holds for a family of decreasing functions. Now, we give the proof of Theorem~\ref{lemma1} by choosing appropriate $L$ and $\calg$.

\begin{proof} [Proof of Theorem~\ref{lemma1}-(a)] 
Fix a pattern $v\in S_d$. For each permutation $\pi \in S_n,$ we define a function $g_{\pi,v}(.):\calp(n)\to \rr_+$ as 
\beqn \label{g_pi_1}
g_{\pi,v}(A) = 
     \begin{cases}
       1 &\quad\text{if } \pi(A) \mbox{ avoids } v \\
       0 &\quad\text{otherwise.} \notag \\ 
     \end{cases}
\feqn
For each $\pi\in S_n,$ $g_{\pi,v}(A)$ is an decreasing function on the distributive lattice $L=\calp(n).$ This is clear from the fact that for any $A\subset B\subset [n]$, we have
\beq
g_{\pi,v}(A) = [\pi(A) \mbox{ avoids } v] \geq [\pi(B) \mbox{ avoids } v] = g_{\pi,v}(B),
\feq
where $[h]$ is one if $h$ holds true, and is zero otherwise. Let $\calg=\{ g_{\pi,v}\ | \ \pi\in S_n \},$ and apply the FKG inequality \eqref{fkg_ineq}. To that goal, choose any $A\subset [n].$ If $|A|\geq d,$ one can find $\pi \in S_n$ where $\pi(A)$ contains $v$ and hence $\prod_{\pi\in S_n}g_{\pi,v}(A) = 0.$ If $|A|\leq d-1,$ then for any $\pi\in S_n,$ $\pi(A)$ avoids $v$ and hence $\prod_{\pi\in S_n}g_{\pi,v}(A) = 1.$ Therefore,
\beqn \label{eq1}
\prod_{\pi\in S_n}g_{\pi,v}(A) = 
     \begin{cases}
       1 &\quad\text{if}\quad |A|\leq d-1 \\
       0 &\quad\text{otherwise.} 
     \end{cases}
\feqn
Next, recall $S_n$ can be written as
\beqn \label{Sn_part}
S_n = F_0^v(S_n) \cup \left( \cup_{\ell\geq 1} \cup_{\mbox{distinct }B_1,\cdots,B_\ell\in [n]_d}F_{\{B_1,\cdots, B_\ell\}}^v(S_n)\right).
\feqn 
Observe that if $\pi$ avoids $v$, then $g_{\pi,v}(A)=1$ for all $A\in \calp(n)$ and hence
\beqn \label{eq1_thm1}
\sum_{A\in \calp(n)} \mu(A) g_{\pi,v}(A) = \sum_{A\in \calp(n)} \mu(A) = 1.
\feqn
Similarly, given any $\ell$ distinct $B_1,\cdots,B_\ell\in[n]_d$, and for any $\pi\in F_{\{B_1,\cdots, B_\ell\}}^v(S_n)$, we have  
\beqn \label{eq2_thm1}
\sum_{A\in \calp(n)}\mu(A)g_{\pi,v}(A) &=& \mu \left( A\in \calp(n) | \pi(A) \mbox{ avoids } v \right) \notag \\
&=& \mu\left( A\in \calp(n) | B_i \not\subset A, 1\leq i \leq \ell \right).
\feqn 
Finally, we plug \eqref{eq1_thm1} and \eqref{eq2_thm1} into the LHS and \eqref{eq1} into the RHS of \eqref{fkg_ineq}, and use \eqref{Sn_part} to group the terms. This completes the proof.
\end{proof}

The proof of the second part follows a similar argument as the first one. However, there is a difference in the choice of $\calg$ and $L"$.

\begin{proof} [Proof of Theorem~\ref{lemma1}-(b)]
Set $T_d := \{A_i \ | \  0\leq i \leq d \}.$ Observe $T_d$ is a distributive lattice and that any probability measure $\mu$ whose support is $T_d$ is indeed log-submodular. This is obvious given that for any $i<j$, we have $A_i\subsetneq A_j$ and hence
\beq
\mu(A_i)\mu(A_j) =  \mu(A_i\cap A_j )\mu(A_i\cup A_j).
\feq

For a given $\pi\in S_n,$ we define the function $g_{\pi,v}:T_d \to \{0,1\}$ as
\beqn \label{g_pi_2}
g_{\pi, v}(A_i) = 
     \begin{cases}
       1 &\quad i< \ell_{\pi, v, T_d} \\
       x_i &\quad i \geq \ell_{\pi, v, T_d}  \notag
     \end{cases},
\feqn
for any $A_i\in T_d,$ where $\ell_{\pi, v, T_d}$ is the minimal value between $1$ and $d$ where $\pi$ avoids $v^\ell.$ $\ell_{\pi, v, T_d}$ is set to infinity when $\pi$ contains $v.$ We first show for any $\pi$, $g_{\pi,v}$ is decreasing on $T_d$. To that goal, let $i<j$: 
\begin{itemize}
    \item If $\ell_{\pi, v, T_d} \leq i < j$, $g_{\pi,v}(A_i)=x_i  \geq  g_{\pi,v}(A_j)=x_j.$
    \item If $i < j < \ell_{\pi, v, T_d},$ $ g_{\pi,v}(A_i)=g_{\pi,v}(A_j)=1.$
    \item If $i < \ell_{\pi, v, T_d} \leq j,$ $1= g_{\pi,v}(A_i)>g_{\pi,v}(A_j)=x_j.$
\end{itemize}
Hence, we could apply FKG inequality \eqref{fkg_ineq} with $L=T_d$ and $\calg=\{ g_{\pi, v} \ | \ \pi\in F_0^v(S_n) \}$. To obtain the RHS of \eqref{fkg_ineq}, observe that, given any $A_i\in T_d$, we have 
\beqn \label{rhs_in1}
\prod_{\pi\in S_n}g_{\pi,v}(A_i) = \prod_{\ell=2}^{i} x_\ell^{f_0^{v^{\ell}|v^{\ell-1}}(S_n)}.
\feqn
To calculate the LHS of \eqref{fkg_ineq}, pick any $\pi\in F_0^v(S_n)$. In this case,
\beq
\sum_{i=0}^d \mu(A_i)g_{\pi,v}(A_i) = \sum_{i=0}^{\ell_{\pi,v,T_d}-1} \mu(A_i) + \sum_{i=\ell_{\pi,v,T_d}}^d x_i\mu(A_i). 
\feq
Given that
\beq
F_0^v(S_n) = \cup_{\ell=2}^d F_0^{v^\ell | v^{\ell-1}}(S_n),
\feq
the LHS of the FKG inequality becomes
\beqn \label{lhs_in1}
\prod_{\ell=2}^d \left(\sum_{i=0}^{\ell-1} \mu(A_i) + \sum_{i=\ell}^d x_i\mu(A_i)\right)^{f_0^{v^\ell|v^{\ell-1}}(S_n)}. 
\feqn
Inserting \eqref{rhs_in1} and \eqref{lhs_in1} into \eqref{fkg_ineq} completes the proof.
\end{proof}

Let $\mu$ and $\nu_{\calb}$ be as before. We follow a similar argument as part (a). This time however we set $\pi\in S_n$ to be a fixed permutation. we choose $L=\calp(n)$ and  $\calg=\{ g_{\pi,v}\ | \ v\in S_d \}$, and apply the FKG inequality \eqref{fkg_ineq}. To that end, pick any $A\subset [n].$ If $|A|\geq d,$ one can find $v \in S_d$ where $\pi(A)$ contains $v$ and hence $\prod_{\pi\in S_n}g_{\pi,v}(A) = 0.$ If $|A|\leq d-1,$ then for any $v\in S_d,$ $\pi(A)$ avoids $v$ and hence $\prod_{\pi\in S_n}g_{\pi,v}(A) = 1.$ Therefore,
\beqn \label{thm2-ineq1}
\prod_{v\in S_d}g_{\pi,v}(A) = 
     \begin{cases}
       1 &\quad\text{if}\quad |A|\leq d-1 \\
       0 &\quad\text{otherwise.} \\ 
     \end{cases}
\feqn
Next, observe that if $\pi$ avoids $v$, then $g_{\pi,v}(A)=1$ for all $A\in \calp(n)$ and hence
\beqn \label{thm2-ineq2}
\sum_{A\in \calp(n)} \mu(A) g_{\pi,v}(A) = \sum_{A\in \calp(n)} \mu(A) = 1.
\feqn
However, for $v\in C_d(\pi),$ we have 
\beqn \label{thm2-ineq3}
\sum_{A\in \calp(n)}\mu(A)g_{\pi,v}(A) &=& \mu \left( A\in \calp(n) | \pi(A) \mbox{ avoids } v \right) \notag \\
&=& \mu\left( A\in \calp(n) | \forall B\in \calb_\pi(v), \ B \not\subset A \right). 
\feqn
Plugging \eqref{thm2-ineq1}-\eqref{thm2-ineq3} into \eqref{fkg_ineq} completes the following result
\begin{lemma}
For each fixed $\pi\in S_n,$ we have $$\prod_{v\in  C_d(\pi)} \nu_{\calb_\pi(v)} \leq \sum_{A\subset [n], |A|<d} \mu(A).$$
\end{lemma}

\end{document}